\documentclass{amsart}

\usepackage{amsmath, amssymb, amsthm} 

\theoremstyle{plain}                       
\newtheorem{theorem}{Theorem}          
\newtheorem{lemma}[theorem]{Lemma}

\theoremstyle{remark}                      
\newtheorem*{acknow}{Acknowledgements}          

\newcommand\pair[1]{\langle #1\rangle}
\newcommand\ppair[1]{(#1)}

\DeclareMathOperator\supp{supp}
\DeclareMathOperator\linsp{span}
\DeclareMathOperator\clsp{\overline{\linsp}}

\newcommand{\lt}{L}                        
\newcommand{\rt}{R}

\newcommand\col{\colon}
\newcommand\sub{\subseteq}
\newcommand\bus{\supseteq}

\newcommand{\complex}{\mathbb{C}}          

\newcommand{\A}{\mathrm{A}}              
\newcommand{\B}{\mathrm{B}}              
\newcommand{\Bred}{\mathrm{B}_r}              

\newcommand{\set}[2]{\{\,{\textstyle#1};\,{\textstyle #2}\,\}}

\newcommand{\inv}{^{-1}}                   
\newcommand\conj{\overline}                

\newcommand\lone{\mathrm{L}^1}
\newcommand\ltwo{\mathrm{L}^2}
\newcommand\linfty{\mathrm{L}^\infty}
\newcommand{\vn}{\mathrm{VN}}
\newcommand\cstarred{\mathrm{C}^*_r}

\newcommand{\C}{\mathrm{C}}
\newcommand{\ucb}{\mathrm{UCB}}

\newcommand{\cop}{\Gamma}
\newcommand{\id}{\mathrm{id}}
\newcommand{\M}{\mathrm{M}}

\newcommand\ot{\otimes}

\newcommand\dual[1]{\widehat{#1}}

\begin{document}

\title{Subgroups and strictly closed invariant C*-subalgebras}

\author{Pekka Salmi}

\address{Deparment of Mathematical Sciences, 
University of Oulu, PL~3000, FI-90100 Oulun yliopisto, Finland}

\email{pekka.salmi@iki.fi}

\subjclass[2010]{Primary 46L99, Secondary 22D15, 22D25}

\keywords{locally compact group, strict topology, multiplier algebra, 
left invariant, translation invariant, group C*-algebra}

\begin{abstract}
We characterise the strictly closed left invariant C*-subalgebras 
of the C*-algebra $\C_b(G)$ of bounded continuous functions 
on a locally compact group $G$. On the dual side, 
we characterise the strictly closed invariant C*-subalgebras 
of the multiplier algebra of the reduced group C*-algebra $\C_r^*(G)$ 
when $G$ is amenable. In both cases, these C*-subalgebras correspond to 
closed subgroups of $G$.
\end{abstract}

\maketitle

\section{Introduction}

Takesaki and Tatsuuma proved in \cite{takesaki-tatsuuma:duality-subgroups}
that left translation invariant von Neumann subalgebras 
of $\linfty(G)$ for  a locally compact group $G$  
correspond to closed subgroups $H$ of $G$. Namely, 
the left invariant subalgebra is
formed by functions in $\linfty(G)$
that are constant on right cosets of $H$. 
Takesaki and Tatsuuma proved also the dual result:
invariant von Neumann subalgebras of 
the group von Neumann algebra $\vn(G)$ 
correspond to closed subgroups of $G$.
In this case, the subalgebra consists of operators in $\vn(G)$ supported by the 
corresponding closed subgroup. 

In the C*-algebraic setting, Lau and Losert 
proved in \cite{lau-losert:complemented} 
that left translation invariant C*-subalgebras of C*-algebra $\C_0(G)$ 
of continuous functions on $G$  vanishing at infinity
correspond  to \emph{compact} subgroups $H$. 
Similarly as in the case of $\linfty(G)$, 
the subalgebra consists of the functions
in $\C_0(G)$ that are constants on right cosets of $H$. 
The dual version of this result is shown in \cite{salmi:compact-subgroups}
for amenable locally compact groups:
invariant C*-subalgebras of the reduced group C*-algebra $\C_r^*(G)$ 
correspond to \emph{open} subgroups of $G$, the C*-subalgebra 
being the collection of elements in $\C_r^*(G)$ 
supported by the open subgroup.

We see that if we look at the C*-algebraic side of things,
there is the constraint that the subgroups are either
compact or open (which are dual to each other).  
To go beyond these constraints, we must look into the multiplier 
algebras. That is, we replace $\C_0(G)$ and $\cstarred(G)$ with their
multiplier algebras; in the former case this is the 
C*-algebra $\C_b(G)$ of all bounded continuous functions on $G$.
However, these multiplier algebras  contain lots of invariant 
C*-subalgebras that are not associated with subgroups:
for example, the C*-subalgebra of weakly almost periodic functions in 
$\C_b(G)$. 
If we want an association with subgroups, we should consider 
invariant C*-subalgebras that are closed under the 
strict topology.
Recall that the strict topology on the multiplier algebra $\M(A)$ 
of a C*-algebra $A$ is the topology 
generated by the seminorms $x\mapsto\| xa\|$, 
$x\mapsto \|ax\|$, where $a$ runs through the elements of $A$. 

In this note, we shall prove that there is a one-to-one correspondence 
between closed subgroups of a locally compact group $G$ 
and strictly closed left invariant C*-subalgebras of $\C_b(G)$.
We shall also prove that for amenable $G$ 
there is a one-to-one correspondence between closed subgroups 
and strictly closed invariant C*-subalgebras of 
the multiplier algebra of the reduced group C*-algebra $\C_r^*(G)$.
These results show that the strict topology, rather than the norm
topology, is often the right topology for C*-algebras.

\section{Stricly closed left invariant C*-subalgebras of 
         bounded continuous functions}

In this section we characterise the strictly closed left invariant 
C*-subalgebras of $\C_b(G)$ where $G$ is a locally compact group.
These correspond to the closed subgroups of $G$. 
The proof follows along the same lines as that of 
\cite[Lemma~12]{lau-losert:complemented}, but of course 
the use of strict topology requires some subtlety.

We let $\lt_s$ and $\rt_s$ denote the left and right translation
operators: $\lt_sf(t) = f(st)$, $\rt_s f(t) = f(ts)$ where
$s,t\in G$ and $f\col G\to \complex$.

\begin{theorem}
Let $G$ be a locally compact group. 
There is  a one-to-one correspondence between 
closed subgroups $H$ of $G$ and 
strictly closed, left invariant C*-subalgebras $X$ of $\C_b(G)$:
\begin{align}
X &= \set{f\in \C_b(G)}{\rt_s f = f \text{ for every }s\in H} \label{eq:X} \\
H &= \set{s\in G}{\rt_s f = f \text{ for every }f\in X}.  \label{eq:H}
\end{align}
Moreover, $H$ is normal if and only if the corresponding $X$ 
is right invariant.
\end{theorem}

\begin{proof}
Suppose that $H$ is a closed subgroup of $G$ and 
let $X$ be defined by \eqref{eq:X}. 
Then $X$ is obviously a left invariant C*-subalgebra of $\C_b(G)$. 
Suppose $f$ is in the strict closure of $X$ in $\C_b(G)$. 
Since strict convergence implies pointwise convergence,
it follows that also $f$ satisfies 
$\rt_s f = f$. So $X$ is strictly closed.

Conversely, suppose that a strictly closed, left invariant C*-subalgebra
$X$ is given and define $H$ by \eqref{eq:H}. 
Then $H$ is a closed subgroup of $G$. 
Let $G/H$ denote the homogeneous space of right cosets of $H$ 
and define 
\[
\pi\col \C_b(G/H)\to \C_b(G), \qquad \pi(f)(s) = f(sH).
\]
Using the fact that on bounded sets the strict topology 
agrees with the compact--open topology \cite{buck:strict}, 
it is easy to see that $\pi$ is strictly continuous on
bounded sets.
Taylor has shown in \cite[Corollary 2.7]{taylor:strict}
that the strongest locally convex topology agreeing with the 
strict topology on norm-bounded sets is actually the 
strict topology itself.
Consequently, a linear map from a multiplier algebra to 
a locally convex space is strictly continuous if it is strictly 
continuous on bounded sets. 
Therefore, $\pi$ is in fact strictly continuous.
Moreover, $\pi$ is a $*$-isomorphism from $\C_b(G/H)$ onto 
\[
Y := \set{f\in \C_b(G)}{\rt_s f = f \text{ for every }s\in H}.
\]
Obviously $X\sub Y$. 
To show that $X = Y$, 
we shall apply the strict Stone--Weierstrass theorem due 
to Glicksberg~\cite{glicksberg:strict-stone-weierstrass}:
if a strictly closed C*-subalgebra $X$ of $\C_b(\Omega)$
separates points of the locally compact space $\Omega$, 
then $X = \C_b(\Omega)$.

First of all, $\pi\inv(X)$ is a strictly closed C*-subalgebra
of $\C_b(G/H)$  (because $\pi$ is strictly continuous).
To show that $\pi\inv(X)$ separates points of $G/H$, let 
$s,t\in G$ with $s\inv t \notin H$. 
By the definition of $H$, there is $f$ in $X$ such that
$f(us\inv t)\ne f(u)$ for some $u$ in $G$. 
Using the left invariance of $X$, we have $f(s)\ne f(t)$
and so $\pi\inv(f)(sH)\ne \pi\inv(f)(tH)$.
Then it follows from  the strict Stone--Weierstrass theorem
that $X = Y$.

This argument also shows that if we start from $X$ and 
define $H$ by \eqref{eq:H}, then $X_H$ derived from $H$ by
\eqref{eq:X} is equal to $X$. 
On the other hand, if we start from $H$ and define $X$
by  \eqref{eq:X}, then $X$ is $*$-isomorphic to $\C_b(G/H)$. 
Now if $s\notin H$, then there is $f$ in $\C_b(G/H)$ such 
that $f(sH) \ne f(H)$. Taking in count the identification of 
$X$ with $\C_b(G/H)$, this means that $s$ is not in the 
subgroup $H_X$ arising from $X$ via \eqref{eq:H}. 
That is, $H = H_X$. 

It remains to show that $H$ is normal if and only if the corresponding
$X$ is right invariant. Now $X$ is right invariant if and only if 
$\rt_h\rt_s f = \rt_s f$ for every $f\in X$, $s\in G$ and $h\in H$. 
But $\rt_h\rt_s f = \rt_s \rt_{s\inv h s} f$ 
so $X$ is right invariant if and only if 
$\rt_{s\inv h s} f = f$ for every $f\in X$, $s\in G$ and $h\in H$. 
The last condition is equivalent with normality of $H$.
\end{proof}

\section{Stricly closed invariant C*-subalgebras of the multiplier 
algebra of the group C*-algebra}

Let $G$ be an amenable locally compact group,
and let $\lambda$ be the left regular representation of $G$ on $\ltwo(G)$.
We shall consider the reduced group C*-algebra $\C_r^*(G)$ 
generated by $\lambda(\lone(G))$.
Of course since we take $G$ amenable, the reduced group C*-algebra
of $G$ is isomorphic with the universal group C*-algebra.
The dual space of $\C_r^*(G)$ is the reduced Fourier--Stieltjes 
algebra $\B_r(G)$, which is a Banach algebra under pointwise 
multiplication of functions
(Eymard introduced the reduced Fourier--Stieltjes algebra 
in \cite{eymard:fourier}, denoting it by $\B_\rho(G)$).
In the amenable case, $\Bred(G)$ coincides with the 
Fourier--Stieltjes algebra $\B(G)$ consisting of 
coefficients of all unitary representations of $G$.

We shall also need the group von Neumann algebra $\vn(G)$,
which is the von Neumann algebra generated by $\lambda$. 
The predual of $\vn(G)$ is the Fourier algebra $\A(G)$
consisting of the coefficient functions of $\lambda$.
The multiplier algebra $\M(\C_r^*(G))$ is identified 
with the idealiser of $\C_r^*(G)$ in $\vn(G)$:
\[
\M(\C_r^*(G)) = \set{x\in\vn(G)}{xa,ax\in\C_r^*(G) 
\text{ for every }a\in \C_r^*(G)}.
\]
We note that elements in $\Bred(G)$ have unique strictly 
continuous extensions to functionals on $\M(\C_r^*(G))$, 
and we shall use these extensions without additional notation.

Define a unitary operator $W$ on $\ltwo(G\times G)$ by
\[
W\xi(s,t) = \xi(ts,t)\qquad(\xi\in\ltwo(G\times G,\, s,t\in G).
\]
Then 
\[
\cop\col x\mapsto W^*(1\ot x)W \qquad (x\in \M(\cstarred(G)))
\]
maps $\M(\cstarred(G))$ to 
$\M(\cstarred(G)\ot\cstarred(G))\sub \B(\ltwo(G\times G))$
(e.g.\ $\cop(\lambda(s)) = \lambda(s)\ot\lambda(s)$ for $s\in G$).
We have an action of $\Bred(G)$ on $\M(\cstarred(G))$ via 
\[
u.x = (u\otimes \id)\cop(x)\qquad 
(u\in\Bred(G),\, x\in\M(\cstarred(G)))
\]
where $u\ot\id$ denotes the unique strictly continuous extension 
of the linear map from $\cstarred(G)\ot\cstarred(G)$ to $\cstarred(G)$
determined by $a\ot b\mapsto u(a)b$.
We say that a C*-subalgebra $X$ of $\M(\C_r^*(G))$ 
is \emph{invariant} if $u.x\in X$ for every $x\in X$ and 
$u\in \B_r(G)$. 
To see that this agrees with the notion of invariance 
used in $\C_b(G)$ (i.e. translation invariance),
note that 
\[
\pair{u.x, v} = \pair{x, uv} \qquad(x\in\M(\C_r^*(G)),\, u,v\in\B_r(G)).
\]
Now if $G$ is abelian, 
$\M(\C_r^*(G))\cong \C_b(\dual G)$ and $X\sub \C_b(\dual G)$ is
invariant if and only if it is translation invariant
(if $u = \chi$ is a character and $x$ corresponds to 
a continuous function $f\in \C_b(\dual G)$, 
then $u.x$ corresponds to the translation of $f$ by $\chi$). 

Since $G$ is amenable, there exists a \emph{summing net} 
$\{F_\alpha\}$ that satisfies the following properties
\begin{itemize}
\item each $F_\alpha$ is nonnull and compact
\item $F_\alpha\sub F_\beta$ if $\alpha \le \beta$
\item $G = \bigcup_{\alpha} F_\alpha$
\item $|sF_\alpha\triangle F_\alpha|/|F_\alpha|\to 0$ uniformly on compact sets
\end{itemize}
(see \cite[Theorem 4.16]{paterson:amenability}).
Here we use $|\cdot|$ to denote the left Haar measure of a set
and $\triangle$ to denote the symmetric difference of sets.

Put $\zeta_\alpha = |F_\alpha|^{-1/2}\, 1_{F_\alpha}\in\ltwo(G)$
and $u_\alpha= \zeta_\alpha * \check\zeta_\alpha$,
where $\check\zeta_\alpha(s) = \zeta_\alpha(s\inv)$,
so that $\pair{u_\alpha, x} = \ppair{x\zeta_\alpha \mid \zeta_\alpha}$
for every $x\in\cstarred(G)$.  
Then each $u_\alpha$ is compactly supported and 
$(u_\alpha)$ is a bounded approximate identity in $\A(G)$.

\begin{lemma} \label{lemma:strict-approx}
For every $x\in \M(\cstarred(G))$, 
$u_\alpha. x \to x$ in the strict topology.
\end{lemma}

\begin{proof}
We should show that
\[
(u_\alpha . x)a\to xa \quad\text{and}\quad a(u_\alpha .x)\to ax
\]
in norm for every $a$ in $\cstarred(G)$. 
Consider the first case. Now 
\[
\|(u_\alpha . x)a -  xa\|
\le\|(u_\alpha . x)a - u_\alpha . (xa)\|+\|u_\alpha . (xa) -  xa\|,
\]
and since $xa\in \cstarred(G)$ 
\[
\|u_\alpha . (xa) -  xa\|\to 0.
\]
So we are left to show that
\[
\|(u_\alpha . x)a - u_\alpha . (xa)\|\to 0.
\]
Now suppose first that $a = \lambda(f)$ for some $f$ in $\lone(G)$
with compact support $K$. 
Given $\epsilon>0$, choose $\alpha_0$ such that
\[
\frac{|sF_\alpha\triangle F_\alpha|}{|F_\alpha|}<\epsilon 
\]
for every $s$ in $K$ and $\alpha\ge\alpha_0$.

Now every $v\in \A(G)$ is of the form 
$v = \conj{\xi}*\check\eta$ where $\xi,\eta\in \ltwo(G)$ 
are such that $\|v\| = \|\xi\|_2\|\eta\|_2$. 
Then
\begin{align*}
&|\pair{(u_\alpha . x)\lambda(f) - u_\alpha .(x\lambda(f)), v}|\\
&\quad= \bigl|\bigl(
   W^*(1\otimes x)W(1\otimes\lambda(f))(\zeta_\alpha\otimes \eta) 
  - W^*(1\otimes x\lambda(f))W(\zeta_\alpha\otimes \eta) \bigm|
        \zeta_\alpha\otimes\xi\bigr)\bigr|\\
&\quad= \bigl|\bigl(W(1\otimes\lambda(f))(\zeta_\alpha\otimes \eta) -
   (1\otimes\lambda(f))W(\zeta_\alpha\otimes \eta) \bigm|
    (1\otimes x^*)W(\zeta_\alpha\otimes\xi)\bigr)\bigr|\\
&\quad\le 
\|W (1\otimes\lambda(f))(\zeta_\alpha\otimes \eta)
 - (1\otimes\lambda(f))W(\zeta_\alpha\otimes \eta)\|_2
\|x\|\|\xi\|_2.
\end{align*}
Now
\begin{align*}
&\|W (1\otimes\lambda(f))(\zeta_\alpha\otimes \eta)-
   (1\otimes\lambda(f))W(\zeta_\alpha\otimes \eta)\|_2\\
&\quad=\biggl(\iint\biggl|\int
f(u)\zeta_\alpha(ts)\eta(u\inv t) - f(u)\zeta_\alpha(u\inv ts)\eta(u\inv t)\,
du\biggl|^2ds\,dt \biggl)^{1/2}\\
&\quad\le\biggl(\iint\biggl(\int
|f(u)|\,|\zeta_\alpha(ts) - \zeta_\alpha(u\inv ts)|\,|\eta(u\inv t)|\,
du\biggl)^2ds\,dt \biggl)^{1/2}.
\end{align*}
Plug in 
\[
|\zeta_\alpha(ts) - \zeta_\alpha(u\inv ts)| 
= \frac{1_{u F_\alpha\triangle F_\alpha}(ts)}{|F_\alpha|^{1/2}},
\]
and apply Minkowski's integral inequality. 
Then we have 
\begin{align*}
&\|W (1\otimes\lambda(f))(\zeta_\alpha\otimes \eta)-
   (1\otimes\lambda(f))W(\zeta_\alpha\otimes \eta)\|_2\\
&\qquad\le\int\biggl(\iint
|f(u)|^2\,\frac{1_{u F_\alpha\triangle F_\alpha}(ts)}{|F_\alpha|}\,|\eta(u\inv t)|^2\,
ds\,dt \biggl)^{1/2} du\\
&\qquad= \|\eta\|_2 
  \int |f(u)| \frac{|u F_\alpha\triangle F_\alpha|^{1/2}}{|F_\alpha|^{1/2}}\,du
< \epsilon^{1/2} \|\eta\|_2 \|f\|_1
\end{align*}
for every $\alpha\ge \alpha_0$. Therefore
\[
|\pair{(u_\alpha . x)\lambda(f) - u_\alpha .(x\lambda(f)), v}|
< \epsilon^{1/2} \|f\|_1 \|x\| \|v\|
\]
for every $\alpha\ge \alpha_0$. It follows that
\[
\|(u_\alpha . x)\lambda(f) - u_\alpha . (x\lambda(f))\|\to 0
\]
and by approximation 
\[
\|(u_\alpha . x)a  - u_\alpha . (xa)\|\to 0,
\]
as required. 

That also $a(u_\alpha . x)\to a x$ can be proved similarly. 
\end{proof}

The support $\supp x$ of an operator $x\in\vn(G)$ is 
defined as follows: $s\in G$ is in $\supp x$ if and only 
if for every neighbourhood $U$ of $s$ there is 
$v\in \A(G)$ supported by $U$ such that $\pair{x,v}\ne 0$.

\begin{lemma} \label{lemma:X-strict}
Let $G$ be an amenable locally compact group 
and $H$ a closed subgroup of $G$. Then 
\[
\set{x\in \M(\cstarred(G))}{\supp x\sub H} 
= \clsp{\lambda(H)}
\]
where the $\clsp$ denotes the strictly closed linear span.
\end{lemma}

\begin{proof}
It follows from Proposition 4.8 of \cite{eymard:fourier}
that the set on the left-hand side of the identity 
is strictly closed.
Since $\supp\lambda(h) = \{h\}$ for every $h$ in $H$, we see that
$\clsp\lambda(H)$ is contained in the set 
on the left-hand side of the identity.
 
Conversely, let $x\in \M(\cstarred(G))$ be supported by $H$. 
Suppose first that $x$ is compactly supported. 
Since $x$ is in the double commutant $\lambda(H)''$, there is a bounded net 
$(x_\alpha)_{\alpha\in I}\sub\linsp\lambda(H)$ 
such that $x_\alpha\to x$ in the weak* topology. 
Since $x$ is compactly supported, we may assume without loss of
generality that there is a compact set $K$ that
is a common support for $x$ and for each $x_\alpha$ 
(otherwise we replace $x_\alpha$ with $u.x_\alpha$ where $u$ is 
a compactly supported function in $\A(G)$ with $u=1$ on a
neighbourhood of the support of $x$). 
Let $y\in \cstarred(G)$ be compactly supported. 
Then $\supp (xy)\sub K\supp y$ and $\supp (x_\alpha y)
\sub K\supp y$ (by \cite[Proposition~4.8]{eymard:fourier}). 
Now there is $u$ in $\A(G)$ such that $u = 1$ on 
a neighbourhood of  the compact set $K\supp y$.
Then, for every $v$ in $\B_r(G)$, the function $uv$ is in $\A(G)$ and so  
\[
\pair{x_\alpha y, v}
= \pair{u.(x_\alpha y), v} 
= \pair{x_\alpha y, uv}
\to \pair{xy, uv} = \pair{xy,v}.
\]
Therefore $x_\alpha y \to xy$ weakly in $\cstarred(G)$
for every compactly supported $y$ in $\cstarred(G)$. 
Since the net $(x_\alpha)$ is bounded, 
we get rid of the restriction of compact support by approximation.
We can deal with the other side similarly, so 
for every $y$ in $\cstarred(G)$
\[
x_\alpha y \to xy\text{ weakly and }
yx_\alpha \to yx\text{ weakly.}
\]

We get from weak convergence to norm convergence 
by a typical convexity argument such as the one presented 
in \cite[p. 524]{day:amenable}.
Since there is the subtlety to obtain 
a \emph{single} net $(z_\beta)$ of convex combinations of $x_\alpha$'s 
that works for \emph{every} $y$ in $\cstarred(G)$,
we repeat here the argument of Day \cite{day:amenable}. 

We say that $(x_\beta')$ is a net of convex combinations 
far out in $(x_\alpha)$ if 
for every $\alpha_0$ there is $\beta_0$ such 
that each $x_\beta'$ with $\beta\ge\beta_0$ 
is a convex combination of $x_\alpha$s with $\alpha\ge \alpha_0$.
Let $F = \{y_1, y_2, \ldots, y_n\}$ be a finite subset of $\C^*_r(G)$.
It follows from the Hahn--Banach theorem that 
there is a net $(x_\beta^1)$ 
of convex combinations far out in $(x_\alpha)$
such that $x_\beta^1 y_1\to xy_1$ in norm. 
Now $x_\beta^1 y_2\to xy_2$ weakly since $(x_\beta^1y_2)$ 
is a net of convex combinations far out in $(x_\alpha y_2)$.
Hence there is a net $(x_\gamma^2)$ 
of convex combinations far out in $(x_\beta^1)$
such that $x_\gamma^2 y_j\to xy_j$ in norm for $j=1,2$. 
Inductively, we get a net $(x^F_\beta)$ of convex combinations 
far out in $(x_\alpha)$ such that $x^F_\beta y \to xy$ in norm
for every $y\in F$. 

Let $\mathcal{F}$ be the collection of all finite subsets of
$\C^*_r(G)$, and recall that we denote the index set of $\alpha$'s by
$I$. Then $\mathcal{F}\times I$ is a directed set where 
$(F_1, \alpha_1)\ge (F_2,\alpha_2)$ if and only if
$F_1\bus F_2$ and $\alpha_1\ge \alpha_2$. 
Now for every $(F,\alpha_0)\in\mathcal{F}\times I$ pick $z(F,\alpha_0)$
such that $z(F,\alpha_0)$ is a convex combination 
of elements $x_{\alpha}$ where $\alpha\ge \alpha_0$
and 
\[
\|z(F,\alpha_0)y-xy\|< \frac1{|F|}
\]
for every $y\in F$ (here $|F|$ denotes the cardinality of $F$).
This choice is possible by the construction above.
Then for every $y$ in $\C^*_r(G)$,
the net $(z(F,\alpha)y)_{(F,\alpha)\in \mathcal{F}\times I}$
converges to $xy$ in norm.
Since $(z(F,\alpha))$ is a net of convex combinations far out 
in $(x_\alpha)$, we have $yz(F,\alpha)\to yx$ weakly for every 
$y\in \C^*_r(G)$, so we may repeat the argument 
to obtain a net $(z_\beta)$ of convex hull of $(x_\alpha)$ 
such that $z_\beta\to x$ in the strict topology.
This completes the case when $x$ is compactly supported.

To deal with the case when the support of $x$ is not necessarily compact,
we need amenability. Let $(u_\alpha)$ be the bounded approximate identity 
from Lemma~\ref{lemma:strict-approx}.  
Since each $u_\alpha$ is compactly supported, so is $u_\alpha.x$.
Moreover, since $x$ is supported by $H$, so is each  $u_\alpha.x$.
By Lemma~\ref{lemma:strict-approx}, $u_\alpha.x\to x$ strictly
so we may apply the earlier part of the proof to $u_\alpha.x$'s
and obtain that $x\in\clsp\lambda(H)$.
\end{proof}

\begin{theorem}
Let $G$ be an amenable locally compact group.
There is a one-to-one correspondence between 
closed subgroups $H$ of $G$ and 
strictly closed, invariant C*-subalgebras $X$ of $\M(\cstarred(G))$:
\begin{align}
X &= \set{x\in \M(\cstarred(G))}{\supp x\sub H} \label{eq:X2} \\
H &= \set{s\in G}{\lambda(s)\in X} \label{eq:H2}.
\end{align}
\end{theorem}

\begin{proof}
If $H$ is given, then \eqref{eq:X2} defines a 
strictly closed, invariant C*-subalgebra $X$ of $\M(\C^*(G))$.
Moreover, $X = \clsp\lambda(H)$ by Lemma~\ref{lemma:X-strict}, 
which implies that the application of \eqref{eq:H2} rediscovers $H$. 

Conversely, let $X$ be a strictly closed, invariant C*-subalgebra
of  $\M(\cstarred(G))$, and define $H$ by \eqref{eq:H2}.
Then $H$ is obviously a subgroup and it is not difficult to 
see that $H$ is closed. Indeed, if $f\in\lone(G)$, then 
$s\mapsto \lt_sf$ and $s\mapsto \rt_s f$ are  
continuous maps from $G$ to $\lone(G)$. Therefore 
$s\mapsto \lambda(s)\lambda(f)$ and  $s\mapsto \lambda(f)\lambda(s)$ 
are continuous, and hence $s\mapsto \lambda(s)$ 
is strictly continuous. Since $X$ is strictly closed, it follows 
that $H$ is closed.

Put
\[
Y = \set{x\in \M(\cstarred(G))}{\supp x\sub H}.
\]
It follows from Lemma~\ref{lemma:X-strict} that $Y\sub X$.
On the other hand, let $x\in X$ such that $\supp x$ is 
compact. For every  $s\in \supp x$,
there is a net $(v_\alpha)$ of functions in $\A(G)$ 
such that $v_\alpha. x\to \lambda(s)$ in the weak* topology
(\cite{eymard:fourier}). Now $v_\alpha. x\in X$ by 
invariance. The technique used in the proof of 
Lemma~\ref{lemma:X-strict} allows us to pass from weak* topology
to strict topology, and so it follows that $\lambda(s)$ is
in $X$.  Therefore $\supp x\sub H$; that is, $x\in Y$. 
The general case, when $\supp x$ is not necessarily compact,
is covered by using the bounded approximate identity in $\A(G)$ 
from Lemma~\ref{lemma:strict-approx}.
Hence $Y = X$, which also shows that passing from $X$ to $H$ 
and then applying \eqref{eq:X2} takes us back to~$X$. 
\end{proof}

It seems reasonable to conjecture that given a closed subgroup $H$ 
of an amenable locally compact group $G$, the C*-algebra
\[
X = \set{x\in \M(\cstarred(G))}{\supp x\sub H}
\]
is canonically $*$-isomorphic to $\M(\cstarred(H))$. 
We end this note with some comments on this conjecture. 

First of all, let us embed $\M(\cstarred(H))$ into $\M(\cstarred(G))$. 
There is a normal injective $*$-homomorphism
$\pi\col \vn(H)\to \vn(G)$ which maps each 
$\lambda_H(h)$ to $\lambda_G(h)$ 
(see \cite{takesaki-tatsuuma:duality-subgroups2} and 
\cite{herz:synthesis}).
Restricted to $\M(\cstarred(H))$ this map is 
a strictly continuous injective $*$-homomorphism.
(Indeed, $\pi$ defines a unitary representation of $H$ on $\ltwo(G)$,
which leads to a nondegenerate, faithful representation of 
$\cstarred(H) = \C^*(H)$ on $\ltwo(G)$. 
Due to nondegeneracy, its strict extension  
from $\M(\cstarred(H))$ into $\M(\cstarred(G))$
agrees with the normal $*$-homomorphism $\pi$ that we started with.)

So we have an embedding $\pi\col \M(\cstarred(H))\to \M(\cstarred(G))$.
Since the image of $\pi$ is strictly dense in $X$ 
by Lemma~\ref{lemma:X-strict}, for $\pi\col \M(\cstarred(H)) \to X$ 
to be a $*$-isomorphism, it suffices that the range of $\pi$ 
is strictly closed. Whether that is the case is unknown (to me).
However, the stronger statement that 
$\pi$ is a strict homeomorphism onto its range is certainly false. 
To see this, suppose that $\pi$ 
is a strict homeomorphism onto its range.
Then, by the Hahn--Banach theorem, 
every strictly continuous functional on $\pi(\M(\cstarred(H))$
extends to a strictly continuous functional on $\M(\cstarred(G))$.
That is, for every $u$ in $\B(H)$, there is $v$ in $\B(G)$ such that
$u = v\circ \pi$ as functionals. Evaluating 
at $\lambda(h)$ says that $u(h) = v(h)$.
Therefore every function in $\B(H)$ extends to a $\B(G)$ function.
It is known that this is not true in general, for example 
for the affine group on the real line 
\cite{eymard:fourier,carling:restriction-to-FS}. 

The problem whether the range of $\pi$ is stricly closed 
can be reduced to the elements outside $\ucb(\dual G)$,
the norm closure of compactly supported elements in $\vn(G)$.
Suppose that $(x_\alpha)$ is a net in $\pi(\M(\C_r^*(H)))$ 
converging strictly to $x$ in $\M(C_r^*(G))$.
If $x$ has compact support, then an argument similar 
to the proof of Lemma~\ref{lemma:X-strict} shows that 
$x$ is in fact in the range of $\pi$. 
Approximation in norm extends this observation to all 
$x\in\ucb(\dual G)$. 

\begin{acknow}
This work was insprired by a question posed by Matthias 
Neufang during the special semester on
`Banach algebra and operator space techniques in topological group
theory' at the University of Leeds 
in May/June 2010, funded by the EPSRC grant EP/I002316/1.
I thank the organisers of that semester for generous hospitality. 
I thank Nico Spronk for support during my stay at University 
of Waterloo, where part of this work was carried out. 
\end{acknow}

\providecommand{\bysame}{\leavevmode\hbox to3em{\hrulefill}\thinspace}


\begin{thebibliography}{10}

\bibitem{buck:strict}
R.~C. Buck, \emph{Bounded continuous functions on a locally compact space},
  Michigan Math. J. \textbf{5} (1958), 95--104.

\bibitem{carling:restriction-to-FS}
L.~N. Carling, \emph{On the restriction map of the {F}ourier-{S}tieltjes
  algebra {$B(G)$} and {$B_\rho(G)$}}, J. Funct. Anal. \textbf{25} (1977),
  236--243.

\bibitem{day:amenable}
M.~M. Day, \emph{Amenable semigroups}, Illinois J. Math. \textbf{1} (1957),
  509--544.

\bibitem{eymard:fourier}
P.~Eymard, \emph{L'alg\`ebre de {F}ourier d'un groupe localement compact},
  Bull. Soc. Math. France \textbf{92} (1964), 181--236.

\bibitem{glicksberg:strict-stone-weierstrass}
I.~Glicksberg, \emph{Bishop's generalized {S}tone-{W}eierstrass theorem for the
  strict topology.}, Proc. Amer. Math. Soc. \textbf{14} (1963), 329--333.

\bibitem{herz:synthesis}
C.~S. Herz, \emph{Harmonic synthesis for subgroups}, Ann. Inst. Fourier
  \textbf{23} (1973), 91--123.

\bibitem{lau-losert:complemented}
A.~T.-M. Lau and V.~Losert, \emph{Complementation of certain subspaces of
  {$L\sb \infty(G)$} of a locally compact group}, Pacific J. Math. \textbf{141}
  (1990), 295--310.

\bibitem{paterson:amenability}
A.~L.~T. Paterson, \emph{Amenability}, Mathematical Surveys and Monographs,
  vol.~29, American Mathematical Society, Providence, RI, 1988.

\bibitem{salmi:compact-subgroups}
P.~Salmi, \emph{Compact subgroups and left invariant {C}*-subalgebras of
  locally compact quantum groups}, J. Funct. Anal. \textbf{261} (2011), 1--24.

\bibitem{takesaki-tatsuuma:duality-subgroups}
M.~Takesaki and N.~Tatsuuma, \emph{Duality and subgroups}, Ann. of Math. (2)
  \textbf{93} (1971), 344--364.

\bibitem{takesaki-tatsuuma:duality-subgroups2}
\bysame, \emph{Duality and subgroups. {II}}, J. Funct. Anal. \textbf{11}
  (1972), 184--190.

\bibitem{taylor:strict}
D.~C. Taylor, \emph{The strict topology for double centralizer algebras},
  Trans. Amer. Math. Soc. \textbf{150} (1970), 633--643.

\end{thebibliography}
\end{document}